\def\timestamp{%
Time-stamp: <katowiceproblem.tex: Thursday 06-08-2015 at 14:21:29 (cest)>}
\def\stripname Time-stamp: <#1 #2>{#2}
\edef\filedate{\expandafter\stripname\timestamp}

\documentclass[a4paper]{amsart}

\newcommand\orpr[2]{\langle#1,#2\rangle}

\DeclareMathSymbol\Z  0{AMSb}{`Z}

\DeclareMathSymbol\le    \mathrel{AMSa}{"36}
\DeclareMathSymbol\ge    \mathrel{AMSa}{"3E}

\newcommand\axiom{\mathsf}

\newcommand\GCH{\axiom{GCH}}
\newcommand\ZFC{\axiom{ZFC}}

\newcommand\cee{\mathfrak{c}}
\newcommand\dee{\mathfrak{d}}

\newcommand\cl{\operatorname{cl}}
\newcommand\Gr{\operatorname{Gr}}
\newcommand\calA{\mathcal{A}}
\newcommand\calF{\mathcal{F}}
\newcommand\calI{\mathcal{I}}
\newcommand\calT{\mathcal{T}}
\newcommand\pow{\mathcal{P}}
\newcommand\fin{\mathit{fin}}

\newcommand\functions[2]{\vphantom{#2}^{#1}#2}
\newcommand\preim{^\gets}

\theoremstyle{plain}
\newtheorem*{introthm}{Theorem}
\newtheorem{theorem}{Theorem}[section]
\newtheorem{lemma}[theorem]{Lemma}

\newtheorem{claim}{Claim}[subsection]

\theoremstyle{definition}

\theoremstyle{remark}
\newtheorem*{intro-question}{Question}
\newtheorem{remark}[theorem]{Remark}

\usepackage{amsrefs}

\begin{document}
\title[The Katowice problem]
      {The Katowice problem and autohomeomorphisms of $\omega_0^*$}

\author[Chodounsk\'y]{David Chodounsk\'y\dag}
\address{Institute of Mathematics\\
         AS CR\\
         \v{Z}itn\'a 25\\
         115 67 {} Praha~I\\
         Czech Republic}
\email{david.chodounsky@matfyz.cz} 
\thanks{\dag Research of this author was supported by the GACR project
  I~1921-N25 and RVO:~67985840}

\author[Dow]{Alan Dow\ddag}
\address{
Department of Mathematics\\
UNC-Charlotte\\
9201 University City Blvd. \\
Charlotte, NC 28223-0001}
\email{adow@uncc.edu}
\urladdr{http://www.math.uncc.edu/\~{}adow}
\thanks{\ddag Research of this author was supported by NSF grant 
             No.\ NSF-DMS-0901168.}

\author[Hart]{Klaas Pieter Hart}
\address{Faculty of Electrical Engineering, Mathematics and Computer Science\\
         TU Delft\\
         Postbus 5031\\
         2600~GA {} Delft\\
         the Netherlands}
\email{k.p.hart@tudelft.nl}
\urladdr{http://fa.its.tudelft.nl/\~{}hart}

\author[De Vries]{Harm de Vries}
\address{Haarlemmerstraat 146\\
         2312~GE {} Leiden\\
         the Netherlands}
\email{vries101@gmail.com}



\date{\filedate}

\begin{abstract}
We show that the existence of a homeomorphism between~$\omega_0^*$ 
and~$\omega_1^*$ entails the existence of a non-trivial autohomeomorphism 
of~$\omega_0^*$.  
This answers Problem~441 in~\cite{MR2367385}.

We also discuss the joint consistency of various consequences of $\omega_0^*$
and $\omega_1^*$ being homeomorphic.
\end{abstract}

\dedicatory{The other authors dedicate this paper to Alan,
  who doesn't look a year over 59}

\maketitle

\section*{Introduction}

The Katowice problem, posed by Marian Turza\'nski,
is about \v{C}ech-Stone remainders of discrete spaces.
Let $\kappa$ and~$\lambda$ be two infinite cardinals, endowed with the discrete
topology.
The Katowice problem asks
\begin{quote}
  If the remainders $\kappa^*$ and~$\lambda^*$ are 
  homeomorphic must the cardinals~$\kappa$ and~$\lambda$ be equal?
\end{quote}

Since the weight of $\kappa^*$ is equal to~$2^\kappa$ it is immediate that
the Generalized Continuum Hypothesis implies a yes answer.
In joint work Balcar and Frankiewicz established that the answer is actually 
positive without any additional assumptions, 
\emph{except possibly for the first two infinite cardinals}.
More precisely

\begin{introthm}[\cites{MR0461444,MR511955}]
If $\orpr\kappa\lambda\neq\orpr{\aleph_0}{\aleph_1}$
and $\kappa<\lambda$ then the remainders $\kappa^*$ and~$\lambda^*$ 
are not homeomorphic.
\end{introthm}

This leaves open the following problem.

\begin{intro-question}
Is it consistent that $\omega_0^*$ and~$\omega_1^*$ are homeomorphic?
\end{intro-question}

Through the years various consequences of 
``$\omega_0^*$~and $\omega_1^*$ are homeomorphic''
were collected in the hope that their conjunction would imply $0=1$ and thus 
yield a full positive answer to the Katowice problem.

In the present paper we add another consequence, namely that there
is a non-trivial autohomeomorphism of~$\omega_0^*$.
Whether this is a consequence was asked by Nyikos in~\cite{Nyikos},
right after he mentioned the relatively easy fact that $\omega_1^*$~has a
non-trivial autohomeomorphism if $\omega_0^*$ and $\omega_1^*$
are homeomorphic, see the end of Section~\ref{sec:prelim}.

After some preliminaries in Section~\ref{sec:prelim} we construct our
non-trivial autohomeomorphism of~$\omega_0^*$ in Section~\ref{een}.
In Section~\ref{sec:question} we shall discuss the consequences alluded to
above and formulate a structural question related to them;
Section~\ref{sec:consistency} contains some consistency results regarding 
that structural question.

\section{Preliminaries}
\label{sec:prelim}

We deal with \v{C}ech-Stone compactifications of discrete spaces exclusively.
Probably the most direct way of defining $\beta\kappa$, 
for a cardinal~$\kappa$ with the discrete topology, is as the space
of ultrafilters of the Boolean algebra $\pow(\kappa)$, 
as explained in~\cite{MR776630} for example.

The remainder $\beta\kappa\setminus\kappa$ is denoted $\kappa^*$ and we
extend the notation $A^*$ to denote $\cl A\cap\kappa^*$ for all subsets
of~$\kappa$. 
It is well~known that $\{A^*:A\subseteq\kappa\}$ is exactly 
the family of clopen subsets of~$\kappa^*$.

All relations between sets of the form~$A^*$ translate back to the original
sets by adding the modifier ``modulo finite sets''.
Thus, $A^*=\emptyset$ iff $A$~is finite,
$A^*\subseteq B^*$ iff $A\setminus B$~is finite and so on.

This means that we can also look at our question as an algebraic problem:

\begin{intro-question}
Is it consistent that the Boolean algebras 
$\pow(\omega_0)/\fin$ and $\pow(\omega_1)/\fin$ are isomorphic?
\end{intro-question}

Here $\fin$ denotes the ideal of finite sets.
Indeed, the algebraically inclined reader can interpret $A^*$ as the 
equivalence class of~$A$ in the quotient algebra and read the proof
in Section~\ref{een} below as establishing that there is a non-trivial
automorphism of the Boolean algebra~$\pow(\omega_0)/\fin$.

\subsection{Auto(homeo)morphisms}\label{subsec:autohom}

It is straightforward to define autohomeomorphisms of spaces
of the form~$\kappa^*$:
take a bijection $\sigma:\kappa\to\kappa$ and let it act in the 
obvious way on the set of ultrafilters to get an autohomeomorphism
of~$\beta\kappa$ that leaves $\kappa^*$ invariant.
In fact, if we want to induce a map on $\kappa^*$ it suffices to
take a bijection between cofinite subsets of~$\kappa$.

For example the simple shift $s:n\mapsto n+1$ on~$\omega_0$
determines an autohomeomorphism $s^*$ of~$\omega_0^*$.
We shall call an autohomeomorphism of $\kappa^*$ \emph{trivial} if
it is induced in the above way, otherwise we shall call it non-trivial.

\subsubsection*{A non-trivial autohomeomorphism for $\omega_1^*$}

To give the flavour of the arguments in the next section we prove that
the autohomeomorphism~$s^*$ of~$\omega_0^*$, introduced above,
has no non-trivial invariant clopen sets.
Indeed assume $A\subseteq\omega_0$ is such that $s^*[A^*]=A^*$;
translated back to~$\omega_0$ this means that the symmetric difference
of~$s[A]$ and~$A$ is finite.
Now if $k$~belongs to the symmetric difference then 
either $k\in A\setminus s[A]$ and so $k-1\notin A$
or $k\in s[A]\setminus A$ and so $k-1$ does belong to~$A$.
Conversely, if $k$~is such that $\{k,k+1\}\cap A$ consists of one point
then that point belongs to the symmetric difference of~$A$ and $s[A]$

Now let $K\in\omega_0$ be so large that the symmetric difference
is contained in~$K$.
It follows that for all $k\ge K$ the intersection $\{k,k+1\}\cap A$
consists of zero or two points.
Now consider $\{K,K+1\}\cap A$; if it is empty then, by induction,
so is $\{k,k+1\}\cap A$ for all $k\ge K$, and we conclude that $A$~is finite
and $A^*=\emptyset$.
The opposite case, when $\{K,K+1\}\subseteq A$, leads 
to $\{k:k\ge K\}\subseteq A$ and hence $A^*=\omega_0^*$.

It is an elementary fact about~$\omega_1$ that for every subset~$A$
of~$\omega_1$ and every map $f:A\to\omega_1$ there are uncountably
many $\alpha<\omega_1$ such that $f[A\cap\alpha]\subseteq\alpha$; 
in particular, if $f$~is a bijection between cofinite sets~$A$ and~$B$ one has
$f[A\cap\alpha]=B\cap\alpha$ for arbitrarily large~$\alpha$. 
This then implies that trivial autohomeomorphisms of~$\omega_1^*$
have many non-trivial clopen invariant sets.

And so, if $\omega_0^*$ and $\omega_1^*$ are homeomorphic then
$\omega_1^*$ must have a non-trivial autohomeomorphism.

\section{A non-trivial auto(homeo)morphism}
\label{een}

In this section we prove our main result.
We let $\gamma:\omega_0^*\to\omega_1^*$ be a homeomorphism and use it to construct
a non-trivial autohomeomorphism of~$\omega_0^*$.

We consider the discrete space of cardinality $\aleph_1$
in the guise of $\Z\times\omega_1$.
A natural bijection of this set to itself is the shift to the right,
defined by $\sigma(n,\alpha)=\orpr{n+1}\alpha$.
The restriction, $\sigma^*$, of its \v{C}ech-Stone extension, $\beta\sigma$,
to~$(\Z\times\omega_1)^*$ is an autohomeomorphism. 
We prove that $\rho=\gamma^{-1}\circ\sigma^*\circ\gamma$ is a non-trivial 
autohomeomorphism of~$\omega_0^*$.

To this end we assume there is a bijection $g:A\to B$ between cofinite sets that
induces~$\rho$ and establish a contradiction.

\subsection{Properties of $\sigma^*$ and $(\Z\times\omega_1)^*$}

We define three types of sets that will be useful in the proof:
vertical lines $V_n=\{n\}\times\omega_1$, 
horizontal lines $H_\alpha=\Z\times\{\alpha\}$ and
end~sets $E_\alpha=\Z\times[\alpha,\omega_1)$.

These have the following properties.

\begin{claim}
$\sigma^*[V_n^*]=V_{n+1}^*$ for all~$n$.    \qed
\end{claim}

\begin{claim}
$\{H_\alpha^*:\alpha<\omega_1\}$ is a maximal disjoint family of 
$\sigma^*$-invariant clopen sets.
\end{claim}

\begin{proof}
It is clear that $\sigma^*[H_\alpha^*]=H_\alpha^*$ for all $\alpha$.

To establish maximality of the family let $C\subseteq\Z\times\omega_1$ be 
infinite and such that $C\cap H_\alpha=^*\emptyset$ for all~$\alpha$; 
then $A=\{\alpha:C\cap H_\alpha\neq\emptyset\}$ is infinite.

For each~$\alpha\in A$ let $n_\alpha=\max\{n:\orpr n\alpha\in C\}$;
then $\{\orpr{n_\alpha+1}\alpha:\alpha\in A\}$ is an infinite subset of 
$\sigma[C]\setminus C$, and hence $\sigma^*[C^*]\neq C^*$.  
\end{proof}

\begin{claim}
If $C\subseteq\Z\times\omega_1$ is such that $H_\alpha^*\subseteq C^*$ for 
uncountably many~$\alpha$ then there are a subset~$S$ of~$V_0$ 
such that $S^*\cap E_\alpha^*\neq\emptyset$ for all~$\alpha$
and $(\sigma^*)^n[S^*]\subseteq C^*$ for all but finitely
many~$n$ in~$\Z$. 
\end{claim}

\begin{proof}
For each~$\alpha$ such that $H_\alpha^*\subseteq C^*$ let $F_\alpha$ be the 
finite set $\{n:\orpr n\alpha\notin C\}$. 
There are a fixed finite set~$F$ and an uncountable subset~$A$ of~$\omega_1$ such that 
$F_\alpha=F$ for all $\alpha\in A$; $S=\{0\}\times A$ is as required.  
\end{proof}

\subsection{Translation to $\omega_0$ and $\omega_0^*$}

We choose infinite subsets 
$v_n$ (for $n\in\Z$),
and~$h_\alpha$ and~$e_\alpha$ (for $\alpha\in\omega_1$) 
such that for all~$n$ and~$\alpha$ we have
$v_n^*=\gamma^\gets[V_n^*]$,
$h_\alpha^*=\gamma^\gets[H_\alpha^*]$, and 
$e_\alpha^*=\gamma^\gets[E_\alpha^*]$.

Thus we obtain an almost disjoint family 
$\{v_n:n\in\Z\}\cup\{h_\alpha:\alpha\in\omega_1\}$
with properties analogous to those of the family
$\{V_n:n\in\Z\}\cup\{H_\alpha:\alpha\in\omega_1\}$,
these are

\begin{claim}
$g[v_n]=^*v_{n+1}$ for all $n$. \qed
\end{claim}

\begin{claim}
$\{h_\alpha^*:\alpha<\omega_1\}$ is a maximal disjoint family of 
$g^*$-invariant clopen sets. \qed
\end{claim}

\begin{claim}\label{claim:big-S}
If $c$ is infinite and $h_\alpha\subseteq^*c$ for uncountably many~$\alpha$ 
then there is a subset~$s$ of~$v_0$ such that $s\cap e_\alpha$ is infinite 
for all~$\alpha$ and such that $g^n[s]\subseteq^*c$ for all but finitely 
many~$n$ in~$\Z$. \qed
\end{claim}

\subsection{Orbits of $g$}

By defining finitely many extra values we can assume that one of~$A$ and~$B$
is equal to~$\omega$ and, upon replacing $\sigma$ by its inverse, we may as 
well assume that $A=\omega$.

For $k\in\omega$ we let $I_k=\{n\in\Z:g^n(k)$ is defined$\}$ and 
$O_k=\{g^n(k):n\in I_k\}$ (the orbit of~$k$).

\begin{claim}\label{claim:split}
Each $h_\alpha$ splits only finitely many orbits.  
\end{claim}

\begin{proof}
If $h_\alpha$ splits $O_k$ then there is an $n\in I_k$ such that
$g^n(k)\in h_\alpha$ but (at least) one of $g^{n+1}(k)$ and $g^{n-1}(k)$
is not in~$h_\alpha$.
So either $g^{n+1}(k)\in g[h_\alpha]\setminus h_\alpha$ or
$g^n(k)\in h_\alpha\setminus g[h_\alpha]$.

It follows that each orbit split by~$h_\alpha$ meets the symmetric difference 
of~$g[h_\alpha]$  and~$h_\alpha$; 
as the latter set is finite and orbits are disjoint only finitely many orbits 
can intersect it.
\end{proof}

We divide $\omega$ into two sets: $F$, the union of all finite $g$-orbits,
and $G$, the union of all infinite $g$-orbits.

\begin{claim}\label{claim:meet-infinite}
If $O_k$ is infinite then there are at most two $\alpha$s for which
$O_k\cap h_\alpha$ is infinite.  
\end{claim}

\begin{proof}
First we let $k\in\omega\setminus B$; in this case $I_k=\omega$.
The set $O_k^*$ is $g^*$-invariant, hence $O_k\cap h_\alpha$ is infinite
for some~$\alpha$.
In fact: $O_k\subseteq^*h_\alpha$ (and so $\alpha$~is unique); 
for let $J=\{n: g^n(k)\in h_\alpha$ and $g^{n+1}(k)\notin h_\alpha\}$, then
$\{g^{n+1}(k):n\in J\}\subseteq g[h_\alpha]\setminus h_\alpha$ so that $J$~is
finite.

It follows that the set $X=\bigcup\{O_k:k\in\omega\setminus B\}$ is,
save for a finite set, covered by finitely many of the~$h_\alpha$.

Next let $k\in \omega\setminus X$; in this case $I_k=\Z$
and both sets $\{g^n(k):n<0\}^*$ and $\{g^n(k):n\ge0\}^*$ are $g^*$-invariant.
The argument above applied to both sets yields $\alpha_1$ and $\alpha_2$
(possibly identical) such that $\{g^n(k):n<0\}\subseteq^* h_{\alpha_1}$ and
 $\{g^n(k):n\ge0\}\subseteq^* h_{\alpha_2}$. 
\end{proof}

The following claim is the last step towards our final contradiction.

\begin{claim}
For all but countably many $\alpha$ we have $h_\alpha\subseteq^* F$.  
\end{claim}

\begin{proof}
By Claim~\ref{claim:meet-infinite} the set $D$ of those~$\alpha$ for which 
$h_\alpha$~meets an infinite orbit in an infinite set is countable: 
each such orbit meets at most two~$h_\alpha$s and there are only countably 
many orbits of course.  

If $\alpha\notin D$ then $h_\alpha$ meets every infinite orbit in a finite
set and it splits only finitely many of these, which means that it intersects
only finitely many infinite orbits, and hence that it meets $G$ in a finite set.
\end{proof}

\subsection{The final contradiction}

We now apply Claim~\ref{claim:big-S} to~$F$.
It follows that there is an infinite subset $s$ of~$v_0$ such that
$g^n[s]\subseteq^*F$ for all but finitely many~$n$.
In fact, as $F$ is $g$-invariant one~$n_0$ suffices:
we can then first assume that $g^{n_0}[s]\subseteq F$ 
(drop finitely many points from~$s$) and then use $g$-invariance of $F$ to 
deduce that $g^n[s]\subseteq F$ for all~$n$.

Let $E=\bigcup_{k\in s}O_k$; as a union of orbits this set is $g$-invariant.
There must therefore be an~$\alpha$ such that $E\cap h_\alpha$ is infinite.
Now there are infinitely many~$k\in E$ such that $h_\alpha$ intersects~$O_k$;
by Claim~\ref{claim:split} $h_\alpha$ must contain all but finitely many
of these.
This means that $O_k\subset h_\alpha$ for infinitely many~$k\in s$ and hence
that $h_\alpha\cap v_0$ is infinite, which is a contradiction because
$h_\alpha$ and $v_0$ were assumed to be almost disjoint.

\subsection{An alternative contradiction}

For each $\alpha$ the set $H_\alpha^*$ splits into two minimal
$\sigma^*$-invariant clopen sets, to wit
$\{\orpr n\alpha:n<0\}^*$ and $\{\orpr n\alpha:n\ge0\}^*$ 
(apply the argument in subsection~\ref{subsec:autohom}).
Therefore the same is true for each~$h_\alpha^*$ with respect to~$\rho$.
However, with the notation as above we find uncountably many
$\rho$-invariant clopen subsets of~$h_\alpha^*$, for every
infinite subset~$t$ of~$s$ we can take $(\bigcup_{k\in t}O_k)^*$.

\section{A question}\label{sec:question}

Our result does not settle the Katowice problem but it may point toward a final
solution.
We list the known consequences of the existence of a homeomorphism between
$\omega_0^*$ and~$\omega_1^*$.
\begin{enumerate}
\item $2^{\aleph_0}=2^{\aleph_1}$
\item $\dee=\aleph_1$
\item there is a strong-$Q$-sequence
\item there is a strictly increasing $\omega_1$-sequence~$\mathcal{O}$ of clopen
      sets in~$\omega_0^*$ such that $\bigcup\mathcal{O}$~is dense
      and $\omega_0^*\setminus\bigcup\mathcal{O}$ contains no $P$-points
\end{enumerate}
A strong-$Q$-sequence is a sequence $\langle A_\alpha:\alpha\in\omega_1\rangle$
of infinite subsets of~$\omega$ with the property that for every choice
$\langle x_\alpha:\alpha\in\omega_1\rangle$ of subsets
($x_\alpha\subseteq A_\alpha$)
there is a single subset~$x$ of~$\omega$ such that $x_\alpha=^* A_\alpha\cap x$
for all~$\alpha$.
In~\cite{MR801424} Stepr\=ans showed the consistency of the existence of
strong-$Q$-sequences with~$\ZFC$.

Not only is each of these consequences consistent with~$\ZFC$ but 
in~\cite{david} Chodounsk\'y provides a model where these consequences hold 
simultaneously.

The three structural consequences can all be obtained using the same sets
that we employed in the construction of the non-trivial autohomeomorphism.
We use the sets~$v_n$ to make $\omega$ resemble~$\Z\times\omega$: 
first make them pairwise disjoint and then identify~$v_n$ 
with~$\{n\}\times\omega$ via some bijection between~$\omega$ 
and~$\Z\times\omega$.

Our consequences are now obtained as follows
\begin{enumerate}\setcounter{enumi}{1}
\item For every $\alpha<\omega_1$ define $f_\alpha:\Z\to\omega$ by 
      $f_\alpha(m)=\min\{n:\orpr mn\in e_\alpha\}$; 
      the family $\{f_\alpha:\alpha<\omega_1\}$ witnesses $\dee=\aleph_1$:
      for every $f:\Z\to\omega$ there is an~$\alpha$ such that
      $\{n:f(n)\ge f_\alpha(n)\}$ is finite.
\item The family $\{h_\alpha:\alpha\in\omega_1\}$ is a strong-$Q$-sequence:
       assume a subset $x_\alpha$ of~$h_\alpha$ is given for all~$\alpha$;
       then there is a single subset~$x$ of~$\omega$ such that
       $x^*\cap h_\alpha^*=x_\alpha^*$ for all~$\alpha$.
       To see this take $X_\alpha\subseteq H_\alpha$ such 
       that $X_\alpha^*=\gamma[x_\alpha^*]$ and put $X=\bigcup_\alpha X_\alpha$
       then $X\cap H_\alpha=X_\alpha$ and 
       hence $\gamma\preim[X^*]\cap h_\alpha^*=x_\alpha^*$ for all~$\alpha$.
\item Let $b_\alpha$ be the complement of~$e_\alpha$ and 
      let $B_\alpha$ be the complement of~$E_\alpha$. 
       Then $\langle b_\alpha^*:\alpha<\omega_1\rangle$ is the required
       sequence: in $\omega_1^*$ the complement of~$\bigcup_\alpha B_\alpha^*$
       consists of the uniform ultrafilters on~$\omega_1$;
       none of these is a P-point.
\end{enumerate}
To this list we can now add the existence of a non-trivial 
auto(homeo)morphism~$\rho$ and a disjoint family $\{v_n:n\in\Z\}$ of 
infinite subsets of~$\omega_0$ such that
\begin{enumerate}\setcounter{enumi}{4}
\item $\{v_n:n\in\Z\}\cup\{h_\alpha:\alpha<\omega_1\}$ is almost disjoint,
\item $\rho[v_n^*]=v_{n+1}^*$ for all $n$,
\item $\{h_\alpha^*:\alpha<\omega_1\}$ is a maximal disjoint family
      of $\rho$-invariant sets, and
\item for each $\alpha$ the sets $(h_\alpha\cap\bigcup_{n<0}v_n)^*$
      and $(h_\alpha\cap\bigcup_{n\ge0}v_n)^*$ are minimal clopen 
      $\rho$-invariant sets.
\end{enumerate}
Since the family $\{h_\alpha:\alpha<\omega_1\}$ is a strong-$Q$-sequence
one can find for any (uncountable) subset~$A$ of~$\omega_1$ an infinite 
set~$X_A$ such that $h_\alpha\subseteq^* X_A$ if $\alpha\in A$
and $h_\alpha\cap X_A=^*\emptyset$ if $\alpha\notin A$.

Our proof shows that $\rho$~is in fact not trivial on every such set $X_A$
whenever $A$~is uncountable.

\begin{remark}\label{rem:aleph2}
  Consequence (1) above is the equality $2^{\aleph_0}=2^{\aleph_1}$; it does
  not specify the common value any further.
  We can actually assume, without loss of generality, that
  $2^{\aleph_0}=2^{\aleph_1}=\aleph_2$.
  Indeed, one can collapse $2^{\aleph_1}$ to~$\aleph_2$ by adding a Cohen subset
  of~$\omega_2$; this forcing adds no new subsets of~$\omega_1$
  of cardinality~$\aleph_1$ or less, so any isomorphism between
  $\pow(\omega_0)/\fin$ and $\pow(\omega_1)/\fin$ will survive.
\end{remark}

\begin{remark}
  It is straightforward to show that the completions of $\pow(\omega_0)/\fin$
  and $\pow(\omega_1)/\fin$ \emph{are} isomorphic, e.g., by taking maximal
  almost families of countable sets in both~$\pow(\omega_0)$
  and~$\pow(\omega_1)$ of cardinality~$\cee$.
  These represent maximal antichains in the completions consisting of mutually
  isomorphic elements and a global isomorphism will be the result of
  combining the local isomorphisms.
  This argument works for all cardinals~$\kappa$ that satisfy
  $\kappa^{\aleph_0}=\cee$, see~\cite{david}*{Corollary~1.2.7}.
  Thus, it will most likely be the incompleteness properties
  of the algebras that decide the outcome of the Katowice problem.
\end{remark}

\section{Some consistency}
\label{sec:consistency}

To see what is possible consistency-wise we indicate how some of the features
of the edifice that we erected, 
based on the assumption that $\omega_0^*$ and $\omega_1^*$ are homeomorphic,
can occur simultaneously.
For this we consider the ideal~$\calI$ generated by the finite sets together
with the sets $b_\alpha$ (the complements of the sets~$e_\alpha$).
This ideal satisfies the following properties:
\begin{enumerate}
\item $\calI$ is non-meager,\label{ctbl1}
\item $\calI$ intersects every P-point, \label{ctbl2}
\item $\calI$ is generated by the increasing \label{ctbl3}
      tower $\{b_\alpha:\alpha<\omega_1\}$, and 
\item the differences $b_{\alpha+1}\setminus b_\alpha$ form a 
      strong-$Q$-sequence.\label{ctbl4}
\end{enumerate}
We have already established properties~\eqref{ctbl2}, \eqref{ctbl3} 
and~\eqref{ctbl4}.

We are left with property~\eqref{ctbl1}; that $\calI$ must be non-meager
was already known to B.~Balcar and P.~Simon.

We recall that a family of subsets of $\omega$ is said to be meager
if, upon identifying sets with their characteristic functions,
it is meager in the product space~$2^\omega$.

\begin{lemma}\label{countisnonmeager}
$\calI$ is not meager.
\end{lemma}

\begin{proof}
We assume $\calI$~is meager and use a countable cover by closed nowhere
dense sets to construct a sequence $\langle F_n:n\in\omega\rangle$ of
pairwise disjoint finite sets such that for every infinite set~$X$
the set $F_X=\bigcup_{n \in X} F_n$ does not belong to~$\calI$ ---
this means that $\gamma[F_X^*]$ is associated to an uncountable
subset~$G_X$ of~$\Z\times\omega_1$.

Fix a family $\{X_s: s\in\functions{<\omega}{2}\}$ of infinite subsets
of~$\omega$ such that $X_s\supseteq X_t$, and hence $G_{X_s}\supseteq^*G_{X_t}$,
whenever $s\subseteq t$, and $X_s\cap X_t=\emptyset$,
and hence $G_{X_s}\cap G_{X_t}=^*\emptyset$,
whenever $s$ and $t$ are incompatible.
Using this we can fix $\alpha\in\omega_1$ such that all exceptions in the
previous sentence occur in $\Z\times\alpha$.

Therefore the family $\{G_{X_s}\cap E_\alpha:s\in\functions{<\omega}2\}$
satisfies the relations above without the modifier `modulo finite sets'.
This implies that if $n\in\Z$ and $\beta\ge\alpha$ then there is at most one
branch~$y_{n,\beta}$ in the binary tree $\functions{<\omega}2$ such that
$\orpr n\beta\in G_{X_s}$ for all $s\in y_{n,\beta}$.

Now, since $2^{\aleph_0}=2^{\aleph_1}$ there is a branch, $y$, different from
all~$y_{n,\beta}$.
We can take an infinite set $X$ such that $X\subseteq^* X_s$ for all $s\in y$.
This means of course that $G_X$~is uncountable and that 
$G_X\subseteq^* G_{X_s}$ for all $s\in y$
and hence that there is $\beta\ge\alpha$ such that 
$G_X\setminus G_{X_s}\subseteq\Z\times\beta$ for all~$s$.
However, if $\orpr n\gamma\in G_X$ and $\gamma\ge\beta$ then we should have
both $\orpr n\gamma\in\bigcap_{s\in y}G_{X_s}$ by the above
and $\orpr n\gamma\notin\bigcap_{s\in y}G_{X_s}$ because $y\neq y_{n,\gamma}$.
\end{proof}

The methods from~\cite{david} and~\cite{MR2944766}
can be used to establish the consistency
of $\dee=\aleph_1$ with the existence of an ideal with the 
properties~\eqref{ctbl1} through~\eqref{ctbl4} of~$\calI$
--- let us call such an ideal countable-like.
We have the following result, which is Theorem~4.5.1 from~\cite{david}.

\begin{theorem}
It is consistent with $\ZFC$ that $\dee = \aleph_1$ and 
there is countable-like ideal $\calI$ on~$\omega$.
\end{theorem}

\begin{proof}
We start with a model of $\ZFC+\GCH$ and take an increasing
tower $\calT = \{T_\alpha : \alpha \in \omega_1\}$ in $\pow(\omega)$
that generates a non-meager ideal and let $\calA$ denote the almost
disjoint family of differences 
$\{T_{\alpha+1} \setminus T_\alpha :\alpha \in \omega_1 \}$ --- we
write $A_\alpha=T_{\alpha+1}\setminus T_\alpha$.
Because of the~$\GCH$ we can arrange that 
$\{\omega\setminus T_\alpha:\alpha\in\omega_1\}$ generates a P-point,
which more than suffices for our purposes.

We set up an iterated forcing construction, with countable supports,
of proper $\functions\omega\omega$-bounding partial orders that will
produce a model in which $\dee=\aleph_1$ and the ideal~$\calI$ generated
by~$\calT$ is countable-like.
By the $\functions\omega\omega$-bounding property we get $\dee=\aleph_1$
and the non-meagerness of~$\calI$ for free.

To turn $\calA$ into a strong-$Q$-sequence we use guided Grigorieff forcing,
as in~\cite{MR2944766}: given a choice 
$F=\langle F_\alpha:\alpha\in\omega_1\rangle$, where each~$F_\alpha$
is a subset of~$A_\alpha$, we let $\Gr(\calT,F)$ be the partial order
whose elements are functions of the form $p:T_\alpha\to2$, with the property
that $p\preim(1)\cap A_\beta=^* F_\beta$ for all $\beta<\alpha$.
The ordering is by extension: $p\le q$ if $p\supseteq q$.
This partial order is proper and $\functions\omega\omega$-bounding and
if $G$~is generic on~$\Gr(\calT,F)$ then $X=(\bigcup G)\preim(1)$ is such that
$X\cap A_\alpha=^*F_\alpha$ for all~$\alpha$.
As indicated in~\cite{MR2944766}, by appropriate bookkeeping one can 
set up an iteration that turns $\calA$ into a strong-$Q$-sequence.

One can interleave this iteration with one that destroys all P-points;
this establishes property~\eqref{ctbl2} of countable-like ideals
in a particularly strong way.
For every ideal~$\calI$ that is dual to a non-meager P-filter one considers
the `normal' Grigorieff partial order $\Gr(\calI)$ associated to~$\calI$, 
which consists of functions with domain in~$\calI$ and $\{0,1\}$ as codomain.
The power $\Gr(\calI)^\omega$ and proper and $\functions\omega\omega$-bounding
and forcing with it creates countably many sets that prevent the filter 
dual to~$\calI$ from being extended to a P-point, even in further
extensions by proper $\functions\omega\omega$-bounding partial orders.

All bookkeeping can be arranged so that all potential choices for~$\calA$
and all potential non-meager P-filters can be dealt with.
\end{proof}

We end on a cautionary note.
Though the result above raises the hope of building a model in which one
has a structure akin to that in Section~\ref{sec:question},
the construction has the tendency of going completely in the wrong
direction as regards autohomeomorphisms of~$\omega_0^*$.
As explained in Chapter~5 of~\cite{david}, if one has an 
autohomeomorphism~$\varphi$ that is not trivial on any element of the 
filter dual to~$\calI$ then the generic filter on~$\Gr(\calI)$ 
destroys~$\varphi$ in the following sense: there is no possible
value for~$\varphi(X^*)$, where $X=(\bigcup G)\preim(1)$.
The reason is that this value should satisfy 
$\varphi(p\preim(1)^*)\subseteq\varphi(X^*)$ and 
$\varphi(p\preim(0)^*)\cap\varphi(X^*)=\emptyset$ for all~$p\in G$
and a density argument shows that no such set exists in~$V[G]$.

Thus, if things go really wrong one ends up with a model in which for every
non-meager P-filter $\calF$ and every autohomeomorphism there is a member
of~$\calF$ on which $\varphi$~must be trivial.
This would be in contradiction with the last sentence just before
Remark~\ref{rem:aleph2};
moreover, Theorem~5.3.12 in~\cite{david} shows that 
with some extra partial orders this can actually be made to happen.

\end{document}